\documentclass{amsart}
\usepackage{amsthm}
\usepackage{amsfonts}
\usepackage{amssymb}
\usepackage{amsmath}
\usepackage{nccmath}
\usepackage{mathtools}
\usepackage{mathrsfs}
\usepackage{setspace}
\usepackage{enumerate}
\newtheorem{theorem}{Theorem}[section]

\theoremstyle{definition}
\newtheorem{defn}{Definition}
\theoremstyle{definition}
\newtheorem{rem}{Remark}
\theoremstyle{Remark}

\theoremstyle{proposition}
\newtheorem{prop}{Proposition}[section]

\begin{document}
\title[m-Pluripotential Theory]
 {m-Pluripotential Theory on Riemannian Spaces and Tropical Geometry}

\author{S\.{I}bel \c{S}ah\.{I}n }

\address{Mimar Sinan Fine Arts University, Mathematics Department}

\email{sibel.sahin@msgsu.edu.tr}


\date{\today}


\dedicatory{Dedicated to Professor Ayd{\i}n Aytuna on the occasion of his 70th birthday}


\begin{abstract}
In this study we extend the concepts of $m$-pluripotential theory to the Riemannian superspace formalism. Since in this setting positive supercurrents and tropical varieties are closely related, we try to understand the relative capacity notion with respect to the intersection of tropical hypersurfaces. Moreover, we generalize the classical quasicontinuity result of Cartan to $m$-subharmonic functions of Riemannian spaces and lastly we introduce the indicators of $m$-subharmonic functions and give a geometric characterization of their Newton numbers.
\end{abstract}

\maketitle
\section*{Introduction}

The classical pluripotential theory is perfectly situated on the border of complex analysis and complex geometry. In the wide literature of pluripotential theory there are various studies on complex Monge-Amp\`{e}re equation and its relation with positive currents and plurisubharmonic functions. In \cite{Lagerberg12}, Lagerberg introduced the concepts of superforms/currents and superspaces which are highly related to the tropical geometry of the Riemannian spaces and then in \cite{BB18}, Berndtsson not only extended this formalism to understand the minimal submanifolds of $\mathbb{R}^n$ but also connected these ideas to the complex pluripotential theory and gave variants of the classical results of complex analysis about positive currents e.g El-Mir-Skoda theorem. In this note we will extend these formal ideas further into another direction namely $m$-pluripotential theory which is introduced by Blocki to study the behaviour of complex Hessian equation. Later in \cite{LuThesis} and \cite{ADFE} the ideas of Blocki were connected to $m$-positive currents. Now in this study we will take this connection into Riemannian superspace setting where $m$-positive closed currents actually give information about the intersection of tropical hypersurfaces.

The organization of this paper is as follows: In Section 1 we give the preliminary definitions and results about superspaces/forms/currents and also introduce the new definitions about $m$-positivity in this setting. In Section 2, we mention the basics of tropical geometry in this superspace setting and following Lagerberg's ideas we investigate the relation between $m$-positivity, positive supercurrents and introduce the relative capacity with respect to tropical varieties. The main result in Section 3 gives us the definition of supercurrents of the type $(dd^{\#}u)^k\wedge\beta^{n-m}\wedge T$ (*) for continuous $m$-subharmonic functions $u$ and then we give a generalization of H.Cartan's well-known result on quasicontinuity of subharmonic functions to $m$-subharmonic functions of Riemannian spaces. In section 4 we show that the definition of the supercurrents of type (*) can actually be extended to the class of locally bounded $m$-subharmonic functions. In the last section, first we introduce the indicators and Newton numbers of $m$-subharmonic functions and then give a geometric characterization  of Newton numbers as the Hausdorff measure of certain sets.

\section{Preliminaries}
In this section we will give the preliminary definitions and some
important results that we will use throughout this study by following the work of \cite{BB18} and \cite{Lagerberg12}.
Let us start with the abstract superspace setting:

Let $E$ be an $n$-dimensional vector space over $\mathbb{R}$ that we can identify with $\mathbb{R}^n$. \emph{Superspace} of $E$ is defined to be $E_s=E_0\oplus E_1$ with $(x_1,\dots, x_n)$ and $(\xi_1,\dots, \xi_n)$ being the coordinates of $E_0$ and $E_1$ respectively.

A \emph{superform} on $E$ is a form of the following structure
$$
a=\sum a_{IJ}(x)dx_{I}\wedge d\xi_{J}
$$ where the coefficients $a_{IJ}$ do not depend on $\xi$ variables. The differential operators over these forms are defined as follows
$$
da=\sum\frac{\partial a_{IJ}}{\partial x_k}dx_k\wedge dx_I\wedge d\xi_J~~d^{\#}a=\sum\frac{\partial a_{IJ}}{\partial x_k}d\xi_k\wedge dx_I\wedge d\xi_J.
$$

A \emph{supercurrent} can be written as
$$
T=\sum T_{IJ}dx_I\wedge d\xi_J
$$ where $T_{IJ}$ are distributions on $E_0$. Let $J$ be the complex structure on $T^*(E_s)$ then
$$
J(dx_i)=d\xi_i~~J(d\xi_i)=-dx_i~~\text{and}~~J(a)=a^{\#}
$$ where $d\xi_i=dx_i^{\#}$.

Now we will define the integration in this superformalism. Given a superform $\alpha= \alpha_0 dx\wedge d\xi$ of bidegree $(n,n)$ where $dx=dx_1\wedge\dots\wedge dx_n$ and $d\xi=dx^{\#}$, we write
$$
\int_{E_s}\alpha=\int_{E_0}\alpha_0 dx\int_{E_1}d\xi.
$$

When the oriantation of $E_0$ is chosen and $\alpha_0$ has a convergent integral, the integral over $E_0$ is well-defined. When $dx_i$ are orthonormal and oriented for the integral over $E_1$ we set
$$
C_n\int_{E_1}d\xi=1
$$
where $C_n=(-1)^{n(n-1)/2}$ is the constant which makes the integral of $\alpha=\alpha_0 dx_1\wedge\dots\wedge dx_n\wedge d\xi$ positive when $\alpha_0$ is positive. It can be observed from the following equality that the value of the superintegral is invariant under the change of orientation:
$$
\int_{\mathbb{R}^{n}_s}\alpha= C_n \int_{\mathbb{R}^{n}_s}\alpha_0 dx\wedge d\xi=\left(\int_{\mathbb{R}^n}\alpha_0\right)\left(C_n\int_{\mathbb{R}^n}d\xi\right)=\int_{\mathbb{R}^n}\alpha_0 dx.
$$

A superform $\alpha$ of bidegree $(n-m, n-m)$ is (weakly) positive if
$$
\alpha\wedge\alpha_1\wedge\alpha_1^{\#}\wedge\dots\wedge\alpha_m\wedge\alpha_m^{\#}\geq0
$$ at every point for any choice of $(1,0)$-forms $\alpha_j$.

A supercurrent $T$ of bidegree $(n-m, n-m)$ is (weakly) positive if
$$
\alpha\wedge\alpha_1\wedge\alpha_1^{\#}\wedge\dots\wedge\alpha_m\wedge\alpha_m^{\#}\geq0
$$ for any choice of compactly supported $(1,0)$-forms $\alpha_j$.

\begin{rem}
We have the following properties for the superforms and supercurrents given above (For more details see \cite{BB18}):

\begin{enumerate}
\item A positive superform of bidegree $(1,1)$ is $\alpha=\sum\alpha_{JK}dx_J\wedge d\xi_K$ and $[\alpha_{JK}]$ is a positive semidefinite matrix.
\item K\"{a}hler form in $\mathbb{R}^n$ is given by the form $\beta\coloneqq\sum dx_J\wedge d\xi_J=(1/2)dd^\#|x|^2$.
\item If $\phi$ is a smooth function on $\mathbb{R}^n$ then $\phi$ is convex if and only if $dd^\#\phi$ is a positive superform.
\item A general but possibly not smooth function $\phi$ is convex if and only if $dd^\#\phi$ is a positive supercurrent.
\end{enumerate}

\end{rem}

In the classical pluripotential theory one of the main problems with a vast literature on it is to understand the positive forms and currents and their relation with the complex Monge-Am\`{e}re equation. In \cite{Blocki}, Blocki initiated the study of $m$-positive forms in relation to the complex Hessian equation. Just as the classical theory of positive forms and currents is strongly related to subharmonic/plurisubharmonic functions, $m$-pluripotential theory is interested in the study of the $m$-subharmonic functions which are in fact subclasses in the spectrum between subharmonic and plurisubharmonic functions, i.e $PSH=P_n\subset\dots\subset P_1=SH$ where $P_i$ is the class of $i$-subharmonic functions. In this study we will consider the notion of $m$-positivity in the broad frame of real superformalism.

\begin{defn}
A $(1,1)$-superform $\alpha$ is m-positive on $E_s$ if at every point of $E_s$ we have $\alpha^j\wedge\beta^{n-j}\geq 0$ for all $j=1,\dots,m$.
\end{defn}
Following the lines of the corresponding complex analysis arguments by \cite{Blocki} we have the following result:

\begin{prop}\label{blocki-conv}
Let $1\leq p\leq m$. If $\alpha_1,\dots,\alpha_p$ are m-positive $(1,1)$-superforms then $\alpha_1\wedge\dots\wedge\alpha_p\wedge\beta^{n-m}\geq 0$.
\end{prop}

\begin{rem}
For a $(1,0)$-superform $\gamma$, $(\gamma\wedge\gamma^\#)^2=0$ holds so the argument directly follows from Garding inequality given in (\cite{Blocki}, pp:1740).
\end{rem}

\begin{defn}
Let $\varphi$ be a $(p,p)$-superform on $E_s$ and $T$ be a supercurrent of bidegree $(p,p)$ on $E_s$. Let $p\leq m\leq n$ then we say
\begin{enumerate}
\item[(i)] $\varphi$ is m-positive on $E_s$ if at every point on $E_s$ one has
$$
\varphi\wedge\beta^{n-m}\wedge\alpha_1\wedge\dots\wedge\alpha_{m-p}\geq 0
$$ for all $\alpha_1,\dots,\alpha_{m-p}$, m-positive $(1,1)$-superforms.

\item[(ii)] $T$ is m-positive if $\langle T,\beta^{n-m}\wedge\varphi\rangle\geq0$ for all $\varphi$ m-positive $(m-p,m-p)$ superform on $E_s$.
\end{enumerate}

\end{defn}

\begin{defn}
A smooth function $\phi$ is called \emph{m-subharmonic} if the $(1,1)$-form $dd^\#\phi$ is $m$-positive at every point of $E_s$.

Then a locally integrable function $u$ is \emph{m-subharmonic} if the supercurrent
$$
dd^{\#}u\wedge\beta^{n-m}\wedge\upsilon_1\wedge\dots\wedge\upsilon_{m-1}\geq 0
$$ for all $m$-positive $(1,1)$-forms $\upsilon_1,\dots,\upsilon_{m-1}$.

\end{defn}

\begin{rem}
Note that in this study although we try to construct $m$-pluripotential theory in the superformalism initiated by \cite{BB18}, there is a basic difference in how we define our main objects of interest, namely $m$-subharmonic functions. As it can be seen from the basic example $u=x^{2}_{1}+x^{2}_{2}-x^{2}_{3}$ in $\mathbb{R}^3$ that $u$ is $2$-subharmonic in Berndtsson sense but not in Blocki's. Hence in our study we follow the idea of \cite{ADFE} in order to give a suitable generalization of the standard notion of $m$-positivity to the superspaces.
\end{rem}

Now we will list the properties of m-subharmonic functions (denoted as $sh_m$ for the rest of the study) that follow directly from the definitions [See \cite{AS12} for the complex analogues of these results]:

\begin{prop}\label{propofmsub}
\begin{enumerate}
\item[(1)] If $u$ is of class $\mathcal{C}^2$ then $u$ is m-subharmonic if and only if $dd^{\#}u$ is m-positive on $E_s$.
\item[(2)] Let $\eta_h$, $(0<h\leq 1)$, be an approximate identity on $\mathbb{R}^n$ then the regularization $u_\eta=u\ast\eta_h$ with $u$ m-subharmonic is m-subharmonic.
\item[(3)] If $u,\upsilon$ are m-subharmonic and $a,b\geq0$ then $au+b\upsilon$ and $\max{u,\upsilon}$ are also m-subharmonic.
\item[(4)] If $(u_\alpha)$ is a family of m-subharmonic functions and $u=\sup u_\alpha$ is upper semicontinuous then $u$ is m-subharmonic.
\item[(5)] Convex functions=$sh_n\subset\dots\subset sh_1$= subharmonic functions.
\end{enumerate}
\end{prop}

\section{Tropical Geometry and Relative Capacity With Respect to Tropical Varieties}

We will start this with the necessary background information about tropical varieties and how their geometry is related to supercurrents (For a detailed treatment of these concepts see \cite{Lagerberg12}). Let $A$ be a finite set in $\mathbb{Z}^n$ and $P=conv(A)$ be the convex hull of $A$ in $\mathbb{R}^n$.
\begin{defn}
A tropical polynomial is a function $f(x)=\max_{\alpha\in A}\left(-\upsilon(\alpha)+\alpha \cdot x\right)$ where $\upsilon:A\rightarrow\mathbb{R}$ is an arbitrary function. For a tropical polynomial the associated tropical hypersurface $V_f$ is the set where the convex function $f$ is not smooth.
\end{defn}

Although the set theoretic intersection of two tropical hypersurfaces need not correspond to the support of a tropical variety (see \cite{Lagerberg12} for an example), the intersection of tropical hypersurfaces can be defined in compliance with the intersection theory of the tropical geometry described in the classical literature as follows:

\begin{defn}
Let $f_1,\dots,f_p$ be tropical polynomials with corresponding tropical hypersurfaces $V_{f_1},\dots,V_{f_p}$ then the intersection of $V_{f_1},\dots,V_{f_p}$ is defined as the $(p,p)$- strongly positive closed supercurrent
\begin{equation}\label{TropVarofcodimp}
V=V_{f_1}\wedge\dots\wedge V_{f_p}\coloneqq dd^{\#}f_1\wedge\dots\wedge dd^{\#}f_p.
\end{equation}
\end{defn}

By (\cite{Lagerberg12}, Prop.4.23) we know that $V$ is a tropical variety of co-dimension $p$ and the support of a strongly positive $(p,p)$-supercurrent whose support has co-dimension $p$ and where one demands that each of the affine pieces should have rational slope.

In the literature of pluripotential theory and analytic function theory several types of capacities were introduced to study the regularity of the sets. One of the most well-known of these capacities is the Monge-Amp\`{e}re capacity introduced by Bedford-Taylor in their seminal work \cite{BT82}. Later this important tool was used in the global theory of Monge-Amp\`{e}re equation over the compact K\"{a}hler manifolds by \cite{GZ05} in connection with other complex variants as Alexander capacity and Tchebychev constants. In \cite{ADFE} the idea of capacity is generalized to the relative $m$-capacity where the capacity of each compact is calculated in association with an $m$-positive closed current. In this current work we are interested in understanding the relative capacity with respect to the closed $m$-positive supercurrents which are closely related to the intersection of tropical hypersurfaces in superspaces. Just as in the classical literature of complex pluripotential theory we will use capacity to generalize a famous result of H. Cartan to show the quasicontinuity of the $m$-subharmonic functions (where the original result of Cartan is about the quasicontinuity of subharmonic functions in other words they are continuous up to a set of zero capacity).

Now using the varieties described above we are going to define the relative capacity associated to tropical varieties:
\begin{defn}
Let $D$ be an open set in $\mathbb{R}^n$ and $V\subset D$ be a tropical variety of co-dimension $p$ given as in (\ref{TropVarofcodimp}). Then the relative capacity associated to $V$ is given by, for any $K\subset D$ compact and $n\geq m\geq p$
\begin{equation}\label{capacity}
cap_{V,m}(K,D)=\sup\left\{\int_{K}(dd^{\#}\phi)^{m-p}\wedge\beta^{n-m}\wedge V,~ \phi\in sh_{m}(D), ~0\leq\phi\leq1\right\}
\end{equation} and for every $E\subset D$
$$
cap_{V,m}(E,D)=\sup\{cap_{V,m}(K,D),~~K~\text{compact in}~E\}.
$$
\end{defn}

Following the original definition of Bedford-Taylor, the capacity mentioned above has clearly the following properties and since the analogues have been proved by many authors in the literature we leave the proofs to the reader:
\begin{prop}
\begin{enumerate}
\item[(i)] If $G$ is a Borel set in $D$ then
$$
cap_{V,m}(G,D)=\sup\left\{\int_{G}(dd^{\#}\phi)^{m-p}\wedge\beta^{n-m}\wedge V,~ \phi\in sh_{m}(D), ~0\leq\phi\leq1\right\}
$$
\item[(ii)] If $G_1\subset G_2$ then $cap_{V,m}(G_1,D)\leq cap_{V,m}(G_2,D)$.
\item[(iii)] If $G_1,G_2,\dots$ are subsets of $D$ then $cap_{V,m}\left(\bigcup_{j=1}^{\infty}G_j,D\right)\leq\sum_{j=1}^{\infty}cap_{V,m}(G_j,D)$
\item[(iv)] If $G_1\subset G_2 \subset\dots$ are Borel sets of $D$ then we have $cap_{V,m}(\bigcup_{j=1}^{\infty}G_j,D)=\lim_{j\rightarrow\infty}cap_{V,m}(G_j,D)$.
\end{enumerate}
\end{prop}

Now we can define the pluripolar sets with respect to a tropical variety:
\begin{defn}
A subset $E\subset D$ is called $(V,m)$-pluripolar in $D$ if $cap_{V,m}(E,D)=0$.
\end{defn}
In \cite{BB18}, the idea of m-pluripolarity is given in relation with m-subharmonic functions as follows:
\begin{defn}
A set $F$ is $m$-polar if there is an m-subharmonic function which is equal to $-\infty$ on $F$.
\end{defn}
In the next result we will give the correlation between these two polarity concepts:
\begin{prop}
Let $V$ be a tropical, linear variety of $\mathbb{R}^n$ with co-dimension $p$. Then a set $G\subset D$ is $(V,m)$-pluripolar if and only if $G\cap V$ is $(m+p-n)$-polar in $V\cap D$.
\end{prop}
\begin{proof}
For a tropical linear variety of $\mathbb{R}^n$, let the associated supercurrent be $[V]_s$ which is defined as
$$
[V]_s=[V]\wedge n^{\#}=n\wedge n^{\#}*dS_V
$$
where $[V]$ is the current of integration on $V$, $n$ is the unit normal form on $V$ and $dS_V$ is the surface measure on $V$. Then since $V$ is linear $[V]_s$ is a closed, positive, symmetric supercurrent. Let $G$ be an open subset of $D$ and $u\in sh_m(D)$ with $0\leq u\leq 1$ then for $i:V\cap D\hookrightarrow D$ being the natural injection we have
$$
\int_{G}[V]_s\wedge\beta^{n-m}\wedge(dd^{\#}u)^{m-p}=\int_{G\cap V}(i^*\beta)^{n-m}\wedge(dd^{\#}u)^{m-n+p}
$$
and then following the same lines of the complex setting argument in \cite{AS12} we deduce that $u_{|_{G\cap V}}=i^{*}(u)$ is $(m-n+p)$-subharmonic hence the result follows.
\end{proof}

\section{Convergence of m-Positive Supercurrents and Quasicontinuity of m-Subharmonic Functions}

As it is very well known the Monge-Amp\`{e}re operator or Hessian operator cannot be defined over arbitrary plurisubharmonic functions. Hence throughout the development of the pluripotential theory the idea was always to define these operators over relatively nice functions with negligible singularities and then to understand the convergence behaviour of the positive currents defined by these nice functions as the functions approach to a more singular one. Generalizing the same scheme in the frame of $m$-positive supercurrents and $m$-subharmonic functions first of all we will try to understand the supercurrents of the form $(dd^{\#}u)^k\wedge\beta^{n-m}\wedge T$ over continuous $m$-subharmonic functions $u$ and then using the quasicontinuity with respect to relative capacity we will generalize the definition of these supercurrents to a broader class in Section 4.

Let $D$ be a smoothly bounded domain in $\mathbb{R}^n$ and $T$ be a closed, $m$-positive supercurrent of bidegree $(p,p)$ defined in a neighborhood of $\overline{D}$.
\begin{theorem}\label{defforcont}
Let $1\leq k\leq m$ and $u_1,\dots,u_k$ be continuous $m$-subharmonic functions in $D$ and $T$ be a closed, $m$-positive supercurrent of bidegree $(p,p)$, ($m-p\geq k$). Then one can inductively define a closed, $m$-positive supercurrent
\begin{equation}\label{dftn-cts}
dd^{\#}u_1\wedge\dots\wedge dd^{\#}u_k\wedge T\wedge\beta^{n-m}
\end{equation}
Moreover the Hessian of a continuous $m$-subharmonic function $u$ is defined as
\begin{equation}\label{hessian}
(dd^{\#}u)^m\wedge\beta^{n-m}
\end{equation}
and the sequence of natural approximants $u_{ij}\searrow u_i$ we have the following convergence of $m$-positive supercurrents
\begin{equation}\label{cnvrgnc-approx}
dd^{\#}u_{1j}\wedge\dots\wedge dd^{\#}u_{kj}\wedge T\wedge\beta^{n-m}\rightarrow dd^{\#}u_{1}\wedge\dots\wedge dd^{\#}u_{k}\wedge T\wedge\beta^{n-m}.
\end{equation}
\end{theorem}
\begin{proof}
From the definition of $m$-positive supercurrents the statement (\ref{dftn-cts}) is trivially true for $k=1$. Now assume that it is true for
$$
dd^{\#}u_{1}\wedge\dots\wedge dd^{\#}u_{k-1}\wedge T\wedge\beta^{n-m}
$$ then for an $(m-k-p,m-k-p)$ smooth, compactly supported superform $\alpha$
$$
[dd^{\#}u_{1}\wedge\dots\wedge dd^{\#}u_{k}\wedge T\wedge\beta^{n-m}](\alpha)=\int_{D_{s}}u_kdd^{\#}u_{1}\wedge\dots\wedge dd^{\#}u_{k-1}\wedge T\wedge\beta^{n-m}\wedge dd^{\#}\alpha
$$
is an element of the dual of $(m-k-p,m-k-p)$ superforms hence it defines the supercurrent $[dd^{\#}u_{1}\wedge\dots\wedge dd^{\#}u_{k}\wedge T\wedge\beta^{n-m}]$.
Now fix an $m$-positive, smooth, compactly supported superform $\gamma$ and let $u_{ij}$ be the natural approximants (i.e the convolution of $u_i$ functions with an approximate identity) $u_{ij}\searrow u_i$. Since for all $i$, $u_i$'s are continuous the convergence is uniform and the convergence for $k-1$ gives
$$
[dd^{\#}u_{1}\wedge\dots\wedge dd^{\#}u_{k}\wedge T\wedge\beta^{n-m}](\gamma)=\int_{D_{s}}u_kdd^{\#}u_{1}\wedge\dots\wedge dd^{\#}u_{k-1}\wedge T\wedge\beta^{n-m}\wedge dd^{\#}\gamma
$$
$$
=\lim_{j\rightarrow\infty}\int_{D_{s}}u_{k}dd^{\#}u_{1j}\wedge\dots\wedge dd^{\#}u_{k-1,j}\wedge T\wedge\beta^{n-m}\wedge dd^{\#}\gamma
$$
$$
=\lim_{j\rightarrow\infty}\lim_{p\rightarrow\infty}\int_{D_{s}}u_{kp}dd^{\#}u_{1j}\wedge\dots\wedge dd^{\#}u_{k-1,j}\wedge T\wedge\beta^{n-m}\wedge dd^{\#}\gamma
$$
and by Proposition \ref{blocki-conv} we have
$$
\int_{D_{s}}u_{kp}dd^{\#}u_{1j}\wedge\dots\wedge dd^{\#}u_{k-1,j}\wedge T\wedge\beta^{n-m}\wedge dd^{\#}\gamma
$$
$$
=\int_{D_{s}}dd^{\#}u_{kp}dd^{\#}u_{1j}\wedge\dots\wedge dd^{\#}u_{k-1,j}\wedge T\wedge\beta^{n-m}\wedge\gamma\geq 0
$$
and then (\ref{cnvrgnc-approx}) follows since the convergence is uniform for $u_{ij}\searrow u_i$.
\end{proof}

In the classical study of Hessian equations in $\mathbb{C}^n$ one of the most important properties of $m$-subharmonic functions is their quasicontinuity i.e $m$-sunharmonic functions are in fact continuous outside of a set of arbitrarily small capacity. Now we are going to show that in our setting $m$-subharmonic functions share this nice property which actually gives us that for an $m$-subharmonic function $u$ and a tropical variety $V$ given in (\ref{TropVarofcodimp}); $u$ is continuous on $V$ except possibly on a subset of small relative capacity:

\begin{theorem}\label{quasicontinuity}
Let $D$ be a smoothly bounded domain in $\mathbb{R}^n$ and $V$ be a tropical variety of co-dimension $p$ defined in $D$ and given as in (\ref{TropVarofcodimp}). If $u$ is an $m$-subharmonic locally bounded function then for $\varepsilon>0$ there exists an open set $G\subset D$ such that $cap_{V,m}(G,D)<\varepsilon$ and $u$ is continuous on $D\setminus G$.
\end{theorem}
\begin{proof}
First of all by following verbatim the proofs given in [Prop 1.2.4 in \cite{LuThesis}] and [Lemma 2 in \cite{ADFE}] we see that if $u_j$ is a sequence of smooth $m$-subharmonic functions converging to a locally bounded $m$-subharmonic function $u$ then $u_j$ converges to $u$ in $cap_{V,m}$ on each $E\subset\subset D$ ($\dagger$). Since $D$ is smoothly bounded we can assume that $u$ is bounded in a neighborhood of $\overline{D}$. For a natural approximation $u_j$ of $u$ we know $u_j$'s are $m$-subharmonic and converge uniformly on compacta to $u$ by Proposition \ref{propofmsub} so by ($\dagger$) for each $k\in \mathbb{N}$, there exists $j_k$ such that
$$
G_k=\{u_{j_k}>u+1/k\}
$$ and $cap_{V,m}(G_k,D)<2^{-k}$ for $n\in \mathbb{N}$ such that $2^{-k}<\varepsilon$. Put $G_n=\bigcup_{k\geq n}G_k$ and since $u_j\rightarrow u$ uniformly on $D\setminus G_n$, $u$ is continuous on $D\setminus G_n$ and $cap_{V,m}(G_n,D)\leq \sum_{k\geq n}cap_{V,m}(G_k,D)<2^{-k}<\varepsilon$.
\end{proof}

\section{Definition of m-positive supercurrents for locally bounded m-subharmonic functions}

In this section we will show that the definition of the $m$-positive supercurrents given for continuous $m$-subharmonic functions in Theorem \ref{defforcont} is also valid for locally bounded $m$-subharmonic functions and the main tool in here will be the quasicontinuity of locally bounded $m$-subharmonic functions:

\begin{theorem}
Let $D$ be a smoothly bounded domain in $\mathbb{R}^n$, $V$ be a tropical variety of co-dimension $p$ defined in $D$ which is given as in (\ref{TropVarofcodimp}) and for $1\leq k\leq m$, $u_1,\dots, u_k$ be locally bounded $m$-subharmonic functions ($m-p\geq k$). Then,
\begin{enumerate}
\item
$$
dd^{\#}u_1\wedge\dots\wedge dd^{\#}u_k\wedge V\wedge\beta^{n-m}[\gamma]
$$
$$
=\int_{D_s}u_kdd^{\#}u_1\wedge\dots\wedge dd^{\#}u_{k-1}\wedge dd^{\#}\gamma\wedge V\wedge\beta^{n-m}
$$
defines an $m$-positive current of bidegree $(n-m+p+k,n-m+p+k)$ where $\gamma$ is an $m$-positive $(m-p-k, m-p-k)$ superform.

\item For the natural approximants $u_{i_j}\rightarrow u_i$, $i=1,\dots, k$ one has
$$
dd^{\#}u_{1_j}\wedge\dots\wedge dd^{\#}u_{k_j}\wedge V\wedge\beta^{n-m}\rightarrow dd^{\#}u_1\wedge\dots\wedge dd^{\#}u_k\wedge V\wedge\beta^{n-m}
$$

\end{enumerate}

\end{theorem}

\begin{proof}
First of all for $k=1$ we can obtain the result trivially from the definition of $m$-positive currents. Now if we take the natural approximants $u_{1_j},\dots,u_{k-1_{j}}$ and $u_{k_t}$ then by Theorem \ref{defforcont} we have
$$
\int_{D_s}u_{k_t}dd^{\#}u_{1_j}\wedge\dots\wedge dd^{\#}u_{k-1_{j}}\wedge V\wedge\beta^{n-m}\wedge dd^{\#}\gamma
$$
$$
=\int_{D_s}dd^{\#}u_{k_t}\wedge dd^{\#}u_{1j}\wedge\dots\wedge dd^{\#}u_{k-1_{j}}\wedge V\wedge\beta^{n-m}\wedge\gamma\geq0
$$ for an $m$-positive $(m-k-p,m-k-p)$ superform $\gamma$ and this gives us as $j\rightarrow\infty$
$$
\int_{D_s}u_{k_t}dd^{\#}u_{1}\wedge\dots\wedge dd^{\#}u_{k-1}\wedge V\wedge\beta^{n-m}\wedge dd^{\#}\gamma\geq0
$$
Hence as $t\rightarrow\infty$ we obtain
$$
\int_{D_s}u_{k}dd^{\#}u_{1}\wedge\dots\wedge dd^{\#}u_{k-1}\wedge V\wedge\beta^{n-m}\wedge dd^{\#}\gamma\geq0
$$
and the supercurrent $[dd^{\#}u_1\wedge\dots\wedge dd^{\#}u_k\wedge V\wedge\beta^{n-m}[\gamma]]$ is $m$-positive.

In part $(2)$ we will use the quasicontinuity of $m$-subharmonic functions given in Theorem \ref{quasicontinuity}. Suppose that the assertion given in $(2)$ is true for $k-1$ then by the quasicontinuity we can find an open set $G\subset D$ such that $cap_{V,m}(G,D)<\varepsilon$ and $u_1\in \mathcal{C}(D\setminus G)$. Take a continuous function $\widehat{u_1}\in \mathcal{C}(D)$ such that $\widehat{u_1}=u_1$ on $D\setminus G$ then, for $k=1,\dots,m$
\begin{equation*}
\begin{split}
\Bigg|\int_{D_s}u_{1_j}dd^{\#}u_{2_j}\wedge\dots\wedge dd^{\#}u_{k_{j}}\wedge V\wedge\beta^{n-m}\wedge\gamma
\\
&-\int_{D_s}u_{1}dd^{\#}u_{2}\wedge\dots\wedge dd^{\#}u_{k}\wedge V\wedge\beta^{n-m}\wedge\gamma\Bigg|
\end{split}
\end{equation*}

$$
\leq\left|\int_{supp\gamma\setminus G}(u_{1_j}-u_1)dd^{\#}u_{2_j}\wedge\dots\wedge dd^{\#}u_{k_{j}}\wedge V\wedge\beta^{n-m}\wedge\gamma\right|
$$
$$
+\left|\int_{D_s}\widehat{u_1}[dd^{\#}u_{2_j}\wedge\dots\wedge dd^{\#}u_{k_{j}}-dd^{\#}u_{2}\wedge\dots\wedge dd^{\#}u_{k}]\wedge V\wedge\beta^{n-m}\wedge\gamma\right|
$$

$$
+\left|\int_{supp\gamma\cap G}(u_{1_j}-u_1)dd^{\#}u_{2_j}\wedge\dots\wedge dd^{\#}u_{k_{j}}\wedge V\wedge\beta^{n-m}\wedge\gamma\right|
$$

$$
+\left|\int_{supp\gamma\cap G}(u_1-\widehat{u_1})[dd^{\#}u_{2_j}\wedge\dots\wedge dd^{\#}u_{k_{j}}-dd^{\#}u_{2}\wedge\dots\wedge dd^{\#}u_{k}]\wedge V\wedge\beta^{n-m}\wedge\gamma\right|
$$
since the first two pieces converge to $0$ because of the uniform convergence of $u_{1_j}\rightarrow u_1$ on $supp\gamma\setminus G$ and the continuity of $\widehat{u_1}$ and the last two pieces are negligible since $cap_{V,m}(G,D)<\varepsilon$. Hence we have the convergence of
$$
u_{1_j}dd^{\#}u_{2_j}\wedge\dots\wedge dd^{\#}u_{k_{j}}\wedge V\wedge\beta^{n-m}\rightarrow u_{1}dd^{\#}u_{2}\wedge\dots\wedge dd^{\#}u_{k}\wedge V\wedge\beta^{n-m}
$$
and the result follows.

\end{proof}

\section{Indicators and Newton Numbers of $m$-subharmonic Functions}

In this section we will first introduce the indicator and the Newton number of an $m$-subharmonic function and then we will give a geometric description of Newton numbers through certain polyhedra. Let us first give the motivation for this consideration by mentioning the corresponding results from complex setting of plurisubharmonic functions and the superspace setting of convex functions:

In \cite{Ras01} Rashkovskii introduced the indicator of a Lelong-class plurisubharmonic function $u\in \mathcal{L}$ as

$$
\Psi_{u,x}(y)=\lim_{R\rightarrow\infty} R^{-1}\sup\{u(z):~~|z_k-x_k|\leq |y_k|^R,~~1\leq k\leq n\}
$$

and showed that (Theorem 3.4, \cite{Ras01}) for $\psi_{u,x}(t)=\Psi_{u,x}(e^{t_1},\dots,e^{t_n})$, $t\in\mathbb{R}$ and $\Theta_{u,x}=\{a\in\mathbb{R}^n:~~\langle a,t\rangle\leq \psi^+_{u,x}(t)~~\forall t\in \mathbb{R}^n\}$,
\begin{equation}\label{resultrashkovskii}
M(\Psi_{u,x};\mathbb{C}^n)=n!Vol(\Theta_{u,x})
\end{equation}
where $M(\Psi_{u,x};\mathbb{C}^n)$ is the total Monge-Amp\`{e}re mass of the indicator. Moreover, if one takes $u=\log|P|$ for a polynomial $P$ then the set $\Theta_{u,x}$ is the Newton polyhedron of $P$ at $\infty$ i.e $\Theta_{u,x}=conv(\omega_0\cup\{0\})$ where
$$
\omega_0=\left\{s\in \mathbb{Z}^{n}_{+}:~~\sum_{j}\left|\frac{\partial^sP_j}{\partial z^s}(0)\right|\neq0\right\}
$$ and the right hand side of (\ref{resultrashkovskii}) is the Newton number of $P$ at $\infty$.

Later in \cite{Lagerberg12} Lagerberg gave an analogue of this argument for the convex functions in the superspace setting. For a convex Lelong class function [for details of this real setting analogue of Lelong-class functions see \cite{Lagerberg12}] $f\in \mathcal{L}$ the associated function is given by
$$
\tilde{f}(x)=\lim_{t\rightarrow\infty}\frac{f(tx)}{t}
$$
and then he showed that if $f_1,\dots, f_n$ are tropical polynomials over $A_1,\dots, A_n$ finite sets of points in $\mathbb{Z}^n$ then
$$
\int_{\mathbb{R}^n\times\mathbb{R}^n}dd^{\#}f_1\wedge\dots\wedge dd^{\#}f_n=n!Vol(Newt(f_1),\dots, Newt(f_n))
$$
where $Newt(f_i)$ is the Newton polytope associated to $f_i$.

\begin{defn}
Let $u$ be a locally bounded $m$-subharmonic function with an isolated singularity at $x\in \mathbb{R}^n$. Then its residual mass at $x$ is given by
$$
\tau(u,x)=(dd^{\#}u)^m\wedge\beta^{n-m}\downharpoonleft_{\{x\}\times\mathbb{R}^n}.
$$
\end{defn}

Similar to complex/convex cases let us define the Lelong-class of $m$-subharmonic functions as
$$
\mathcal{L}_{m}=\{f:\mathbb{R}^n\rightarrow\mathbb{R}^n:~~f(x)\leq C|x|+D,~~f~m-subharmonic,~C\geq0, ~D\in \mathbb{R}\}
$$
and the indicator of $f\in \mathcal{L}_m$ at $x\in \mathbb{R}^n$ as
$$
\Psi_{f,x}(y)=\lim_{R\rightarrow\infty} R^{-1}\sup\{f(t):~~|t_k-x_k|\leq |y_k|^R,~~1\leq k\leq n\}
$$

\begin{rem}
Following the same lines of complex/convex cases we see the following facts about the indicator:
\begin{itemize}
\item[i)] $\Psi_{f,x}\in \mathcal{L}_m$.
\item[ii)] $f(t)\leq \Psi_{f,x}(t-x)+C_x,~~\forall t\in \mathbb{R}^n$.
\end{itemize}
\end{rem}

Now we will generalize the idea of Newton numbers to arbitrary $m$-subharmonic functions:
\begin{defn}
Let $u$ be a locally bounded $m$-subharmonic function with isolated singularity at $x\in \mathbb{R}^n$. Then the Newton number of $u$ at $x$ is defined as
$$
N(u,x)=\tau(\Psi_{u,x};0)
$$
\end{defn}

Now we will give the geometric description of the Newton number of an $m$-subharmonic function:

\begin{theorem}
For an $m$-subharmonic function $u\in \mathcal{L}_m$,
$$
N(u,x)=\tau(\Psi_{u,x};0)=\left(\binom{n}{m}\right)^{-1}\mathcal{H}^{n-m}(\Theta_{u,x})
$$
where $\mathcal{H}^{n-m}$ denotes the $(n-m)$-dimensional Hausdorff measure and
$$
\Theta_{u,x}=\{a\in \mathbb{R}^n:~~\langle a,t\rangle\leq \Psi^{+}_{u,x}(t),~~\forall t\in \mathbb{R}^n\}.
$$
\end{theorem}
\begin{proof}
First of all as a result of the generalized comparison principle for positive supercurrents (Theorem 6, \cite{EZ18}) we have the following equality of the residual mass of $\Psi_{u,x}$ and $\Psi^{+}_{u,x}$:
$$
\tau(\Psi_{u,x};0)=(dd^{\#}\Psi_{u,x})^m\wedge\beta^{n-m}\downharpoonleft_{\{0\}\times\mathbb{R}^n}=(dd^{\#}\Psi^{+}_{u,x})^m\wedge\beta^{n-m}\downharpoonleft_{\{0\}\times\mathbb{R}^n}.
$$
Now note that the Steiner formula for $m$-subharmonic functions [See \cite{ColHug1,ColHug2} for more detailed treatment of the formula] gives that for an $m$-subharmonic function $u$ and a set $\Omega$ we have
$$
F_m[u](\Omega)=\int_{\Omega}(dd^{\#}u)^m\wedge\beta^{n-m}=\left(\binom{n}{m}\right)^{-1}\mathcal{H}^{n-m}((\partial u, \Omega))
$$
where
$$
(\partial u, \Omega)=\bigcup_{t_0\in \Omega}\{a\in \mathbb{R}^n: ~u(t)\geq u(t_0)+\langle a,t-t_0\rangle~~\forall t\in \mathbb{R}^n\}
$$
and in our case this equality turns into
$$
F_m[\Psi^{+}_{u,x}](\{0\}\times\mathbb{R}^n)=\left(\binom{n}{m}\right)^{-1}\mathcal{H}^{n-m}(\Theta_{u,x}).
$$
Hence the Newton number of a given $m$-subharmonic function $u\in \mathcal{L}_m$ is given as
$$
N(u,x)=\left(\binom{n}{m}\right)^{-1}\mathcal{H}^{n-m}(\Theta_{u,x}).
$$
\end{proof}

\end{document}